\documentclass[a4paper,10pt,reqno]{amsart}

\usepackage{amsthm,amsmath,amsfonts,amssymb}
\usepackage{color}

\usepackage{todonotes}
\usepackage{empheq}
\usepackage{cases}

\usepackage{hyperref}
\usepackage{pgf, tikz-cd}
\usetikzlibrary{arrows}
\usepackage{tikzscale}
\usepackage{subfigure}

\definecolor{grey}{rgb}{.7,.7,.7}

\newcommand{\e}{\varepsilon}

\newcommand{\N}{\mathbb{N}}
\newcommand{\de}{\partial}

\renewcommand{\phi}{\varphi}

\newcommand{\bj}{{\mathbf{j}}}
\newcommand{\bi}{{\mathbf{i}}}

\renewcommand{\div}{\mathop{\mathrm{div}}}
\renewcommand{\epsilon}{\varepsilon}

\def\R{\mathbb{R}}

\def\H{\mathcal{H}}
\def\grad{\nabla}
\def\dist{\text{dist}}

\theoremstyle{plain}
\newtheorem{thm}{Theorem}[section]
\newtheorem{lem}[thm]{Lemma}
\newtheorem{prop}[thm]{Proposition}

\theoremstyle{definition}
\newtheorem{defin}[thm]{Definition}

\theoremstyle{remark}
\newtheorem{rem}[thm]{Remark}

\theoremstyle{conjecture}

\title[Two examples of minimal Cheeger sets in the plane]{Two examples of minimal Cheeger sets in the plane}

\author{Gian Paolo Leonardi}
\address{Dipartimento di Scienze Fisiche, Informatiche e Matematiche, Universit{\`a} degli Studi di Modena e Reggio Emilia, Via Campi 213/b, I-41125 Modena, ITALY}
\email{gianpaolo.leonardi@unimore.it}

\author{Giorgio Saracco}
\address{Department Mathematik, Universit\"at Erlangen-N\"urnberg, Cauerst. 11, D-91058 Erlangen, Germany}
\email{saracco@math.fau.de}

\thanks{G.P. Leonardi and G. Saracco have been supported by GNAMPA projects: \textit{Variational problems and geometric measure theory in metric spaces} (2016). G. Saracco has been also supported by the DFG Grant n. GZ:PR 1687/1-1}

\subjclass[2010]{Primary: 49Q10, 53A10. Secondary: 35P15}

\keywords{Cheeger problem; minimal Cheeger set; weak regularity; capillarity}

\begin{document}

\definecolor{qqqqtt}{rgb}{0.,0.,0.2}
\definecolor{qqzzqq}{rgb}{0.2,0.2,0.2}
\definecolor{uuuuuu}{rgb}{0.26666666666666666,0.26666666666666666,0.26666666666666666}
\definecolor{wwqqzz}{rgb}{0.3333333333333333,0.3333333333333333,0.3333333333333333}
\definecolor{qqwuqq}{rgb}{0.12941176470588237,0.12941176470588237,0.12941176470588237}
\definecolor{uuuuuu}{rgb}{0.26666666666666666,0.26666666666666666,0.26666666666666666}
\definecolor{ttqqqq}{rgb}{0.06666666666666667,0.06666666666666667,0.06666666666666667}
\definecolor{wwccff}{rgb}{0.4,0.8,1.}
\definecolor{ttttqq}{rgb}{0.2,0.2,0.}

\begin{abstract}
We construct two minimal Cheeger sets in the Euclidean plane, i.e. unique minimizers of the ratio ``perimeter over area'' among their own measurable subsets. The first one gives a counterexample to the so-called weak regularity property of Cheeger sets, as its perimeter does not coincide with the $1$-dimensional Hausdorff measure of its topological boundary. The second one is a kind of porous set, whose boundary is not locally a graph at many of its points, yet it is a weakly regular open set admitting a unique (up to vertical translations) non--parametric solution to the prescribed mean curvature equation, in the extremal case corresponding to the capillarity for perfectly wetting fluids in zero gravity.
\end{abstract}

 \hspace{-3cm}
 {
 \begin{minipage}[t]{0.6\linewidth}
 \begin{scriptsize}
 \vspace{-2cm}
 This is a pre-print of an article published in \emph{Ann. Mat. Pura Appl.}. The final authenticated version is available online at: http://dx.doi.org/10.1007/s10231-018-0735-y
 \end{scriptsize}
\end{minipage} 
}

\maketitle

\section*{Introduction} 

Given $\Omega \subset \R^n$ open and bounded, its Cheeger constant is defined as
\begin{equation}\label{eq:cheegerconstant}
h(\Omega) :=\inf_{E\subset \Omega} \frac{P(E)}{|E|}\,,
\end{equation}
where $P(E)$ and $|E|$ denote, respectively, the De Giorgi's perimeter and the $n$-dimensional Lebesgue measure of $E$. The variational problem associated with the definition of $h(\Omega)$ first appeared in \cite{Besicovitch, Steiner1841} limitedly to convex subsets of the Euclidean plane; see also \cite{CroftFalconerGuy,SingmasterSouppuris}. A more general formulation is due to Cheeger, who proved in \cite{Cheeger70} that the first eigenvalue of the Laplace-Beltrami operator on a compact Riemannian manifold $M$ is bounded from below by $h^2(M)/4$. Since then, the problem in the Euclidean setting has commonly been known as the \emph{Cheeger problem}. A nice, as well as quite surprising, feature of the Cheeger problem is that it naturally appears in many different contexts such as image processing \cite{AltCasCha2005, BelCasNov2002,CasChaNov2011}, landslide modeling \cite{HIL05, HILR02,IL05}, and fracture mechanics \cite{Kel80}. For further discussion and applications the reader could refer to the surveys \cite{Leo, Parini2011}.  

In particular, the Cheeger problem is closely related to the theory of existence and uniqueness of graph of prescribed mean curvature \cite{Finn1984, Giusti1978, LS16a} which is the cornerstone of the theory of capillary surfaces (a comprehensive treatise is available in \cite{Finn1986}). 
We recall that for an open, bounded and connected set $\Omega \subset \R^n$, a function $u:\Omega \to \R$ is a classical solution to the prescribed mean curvature equation if
\begin{equation}\tag{PMC}\label{eq:PMC}
\div \left(\frac{\grad u(x)}{\sqrt{1+|\grad u(x)|^2}} \right) = H(x),\qquad \forall x\in \Omega,
\end{equation}
for a given Lipschitz function $H$ defined on $\Omega$. Under the assumption of $\mathcal{C}^2$ regularity of $\de \Omega$ (or piece-wise $C^1$ up to a $\H^{n-1}$-negligible set, see \cite[Chapter 6]{Finn1986}), the conditions 
\begin{equation}\label{eq:necandsuf}
\left | \int_A H\, dx\right| < P(A)\, \qquad \forall A\subsetneq \Omega\,,\qquad \qquad \left| \int_\Omega H\, dx \right| =P(\Omega)
\end{equation}
were proved to be necessary and sufficient to existence and uniqueness (up to vertical translations) by Giusti in  \cite{Giusti1978}. Then, whenever $H$ is a positive constant, these necessary and sufficient conditions read equivalently as
\begin{equation}\tag{MC}\label{eq:MC}
\text{$\Omega$ is a minimal Cheeger set\,,}
\end{equation}
that is, $\Omega$ is the unique minimizer of \eqref{eq:cheegerconstant}, and $H = h(\Omega)$.

In \cite{LS16a} we have extended Giusti's results on the existence of solutions to \eqref{eq:PMC} and on the characterization of the extremality condition \eqref{eq:necandsuf} to the class of weakly regular sets $\Omega$. These sets are defined as open bounded sets with finite perimeter such that
\begin{equation}\tag{PH}\label{eq:PH}
P(\Omega) = \H^{n-1}(\de \Omega)
\end{equation}
and
\begin{equation}\label{eq:Poincare}
\min \{P(E; \de \Omega), P(\Omega \setminus E; \de \Omega) \} \le k P(E; \Omega)\,,
\end{equation}
for some $k=k(\Omega)>0$ and for all measurable $E\subset \Omega$. 
We notice that any minimal Cheeger set $\Omega$, for which the intersection $\de\Omega \cap \Omega^{(1)}$ (where $\Omega^{(1)}$ denotes the set of points of density $1$ for $\Omega$) has $\H^{n-1}$-null measure, is weakly regular. This observation, which has been proved by the second author in \cite{Sar17}, provides a sufficient condition for the weak regularity, which can be more easily checked whenever the domain $\Omega$ is a minimal Cheeger set. As a consequence, for any $\Omega$ satisfying \eqref{eq:MC} and $\H^{n-1}(\Omega^{(1)}\cap \de \Omega)=0$, the constant mean curvature problem on $\Omega$, for the ``extremal'' value of the prescribed mean curvature $H=h(\Omega)$, admits a unique solution up to vertical translations. We remark that this extremal situation corresponds to the physical case of capillarity for perfectly wetting fluids in zero gravity. 

It is then natural to ask whether the weak regularity assumption is optimal with respect to the results proved in \cite{LS16a} on the prescribed mean curvature equation. Related to this question is the following one: does \eqref{eq:MC} imply $\H^{n-1}(\Omega^{(1)}\cap \de \Omega)=0$? In the affirmative case, any minimal Cheeger set would be automatically weakly regular. However, in Section \ref{sec:fatCantor} we negatively answer this question by exhibiting a minimal Cheeger set $\Omega_\e$ in the plane (with $\e$ a suitably small parameter) for which \eqref{eq:PH} does not hold. Of course this set needs to be such that $\Omega_\e^{(1)}\cap \de \Omega_\e$ has positive  $\H^{1}$-measure. This is ensured by the presence of a fat Cantor set contained in $\de \Omega_\e$, which is negligible for the perimeter measure. This lack of regularity prevents $\Omega_\e$ from being approximated in measure and perimeter by a sequence of smooth sets that are compactly contained in $\Omega_\e$ (see \cite{Schmidt2014}) and from admitting a trace operator from $BV(\Omega_\e)$ to $L^1(\de \Omega_\e)$ (see \cite[Chapter 9]{Mazya2011}). Indeed, this approximation property and the existence of a suitable trace operator represent two crucial tools used in \cite{LS16a}. However, one might expect that such solutions exist and are unique up to vertical translations, for each one of the two possible values of $H$ that correspond to counting or not the $\H^1$-measure of the fat Cantor set. At the same time, both solutions will become vertical at the reduced boundary of $\Omega_\e$. In conclusion, this example shows that it is not possible to extend the characterization of existence and uniqueness of solutions to \eqref{eq:PMC} given in \cite[Theorem 4.1]{LS16a} by dropping the assumption of weak regularity of the domain (see the discussion after the proof of Theorem \ref{thm:cantor}).

Then, in Section \ref{sec:emmenthal} we build a set $\Omega_{\bold 0}$ that turns out to be a minimal Cheeger satisfying $\H^1(\Omega_{\bold 0}^{(1)}\cap \de \Omega_{\bold 0})=0$, even though its boundary is not regular at all. More precisely, there is a set of positive $\H^1$-measure consisting of points of the reduced boundary of $\Omega_{\bold 0}$, at which $\de \Omega_{\bold 0}$ is not locally a graph. This example is constructed starting from the unitary disk $B_1$ and removing smaller and smaller disks accumulating towards $\de B_1$, so that the resulting set displays a kind of ``porosity''. This example is of interest for two reasons. First, $\Omega_{\bold 0}$ is weakly regular, so that the results of \cite{LS16a} apply (while the previous results due to Giusti and Finn do not) and one deduces the existence and uniqueness up to vertical translations of the solution to \eqref{eq:PMC} in the extremal case of $H(x) = h(\Omega_{\bold 0})$. Second, this example shows the following, quite remarkable fact. On the one hand, a generic small and smooth perturbation of the disk typically produces a dramatic change of the corresponding capillary solution, possibly leading even to a non-existence scenario. On the other hand, the construction of $\Omega_{\bold 0}$ shows that one can produce non-smooth perturbations of a disk that, instead, preserve existence and stability of the capillary solution. Indeed, from this ancestor set, one can build an increasing sequence of minimal Cheeger sets $\Omega_{\bold k}$ converging to the unitary disk both in volume and perimeter, in such a way that the stability result \cite[Proposition 4.4]{LS16a} holds.

Proving a set to be a minimal Cheeger is not an immediate fact. There are results allowing to infer whether a set is a Cheeger set or not and whether it is minimal or not but they apply only in limited circumstances (and limitedly to the plane), as for instance in the case of convex sets \cite{Besicovitch,KawFrid03} or simply connected sets with ``no bottlenecks'' \cite{LNS17}. Given a Cheeger set $E$ in $\Omega$, it is well known that $\de E\cap \Omega$ is an analytic hyper-surface (up to a  closed singular  set of Hausdorff dimension at least $n-8$). Then, our proof of \eqref{eq:MC} is achieved by showing that any Cheeger sets $E$ of $\Omega$ satisfies $\de E \cap \Omega = \emptyset$, which in turns says that the only Cheeger set can be $\Omega$ itself.

\section{Preliminaries}\label{sec:preliminary}

We first introduce some basic notations. We fix $n\ge 2$ and denote by $\R^{n}$ the Euclidean $n$-space. Let $E\subset \R^{n}$, then we denote by $\chi_{E}$ the characteristic function of $E$. For any $x\in \R^{n}$ and $r>0$ we denote by $B_{r}(x)$ the Euclidean open ball of center $x$ and radius $r$. Whenever $x=0$ we shall write $B_r$ instead of $B_{r}(0)$. Given two sets $E,F$, we denote their symmetric difference by $E\Delta F = (E\setminus F)\cup(F\setminus E)$. In order to deal with rescaled sets we introduce the notation $E_{x,r} = r^{-1}(E-x)$, where $E\subset \R^{n}$, $x\in\R^{n}$, and $r>0$. Given an open set $\Omega\subset \R^{n}$ we write $E\subset\subset \Omega$ whenever $E\subset \R^{n}$ is such that its topological closure $\overline E$ is a compact subset of $\Omega$. For any measurable set $E\subset \R^{n}$ we denote by $|E|$ its $n$-dimensional Lebesgue measure. Concerning $n$-dimensional (measurable) sets, we shall identify two such sets $E$ and $F$ as soon as $|E\Delta F|=0$, and write $E=F$ for the sake of brevity. Analogously the inclusions $E\subset F$ should be understood up to null sets.

\begin{defin}[Perimeter]\label{def:per}
Let $E$ be a Borel set in $\R^n$. We define the perimeter of $E$ in an open set $\Omega\subset \R^{n}$ as
\[
P(E; \Omega):=\sup \left\{ \int_{\Omega}\chi_{E}(x)\, \div g(x)\, dx\,:\ g\in C^1_c(\Omega;\, \R^n)\,, \|g\|_{\infty} \leq 1\right\}\,.
\]
We set $P(E) = P(E;\R^{n})$. If $P(E;\Omega)<\infty$ we say that $E$ is a set of finite perimeter in $\Omega$. In this case (see \cite{AFP}) one has that the perimeter of $E$ coincides with the total variation $|D\chi_{E}|$ of the vector--valued Radon measure $D\chi_{E}$ (the distributional gradient of the characteristic function $\chi_{E}$).
\end{defin}

\begin{defin}[P-decomposability]
A set $E\subset \R^n$ of finite perimeter is said to be P-decomposable if there exists a pair of disjoint Borel sets $S$ and $T$, such that $|S|, |T|>0$, $E = S\cup T$, and $P(E) = P(S)+P(T)$. Otherwise, $E$ is said to be P-indecomposable.
\end{defin}

\begin{defin}[Points of density $\alpha$]
Let $E$ be a Borel set in $\R^n$, $x\in \R^n$. If the limit
\[
\theta(E)(x) := \lim_{r\to0^+} \frac{|E\cap B_r(x)|}{\omega_nr^n}
\]
exists, it is called the density of $E$ at $x$. We define the set of points of density $\alpha\in [0,1]$ of $E$ as
\[
E^{(\alpha)} := \left\{ x\in \R^n\,:\, \theta(E)(x) = \alpha\right\}\,.
\]
We also define the essential boundary $\de^{e}E := \R^{n}\setminus (E^{(0)}\cup E^{(1)})$.
\end{defin}

\begin{thm}[De Giorgi Structure Theorem]\label{thm:DeGiorgi}
Let $E$ be a set of finite perimeter and let $\de^*E$ be the reduced boundary of $E$ defined as
\[
\de^* E:=\left \{x\in \de^{e}E\,:\, \lim_{r\to 0^+} \frac{D\chi_E(B_r(x))}{| D\chi_E|(B_r(x))} = -\nu_E(x) \in \mathbb{S}^{n-1} \right\}\,.
\]
Then,
\begin{itemize}
\item[(i)] $\de^{*}E$ is countably $\H^{n-1}$-rectifiable in the sense of Federer~\cite{FedererBOOK};

\item[(ii)] for all $x\in \de^{*}E$, $\chi_{E_{x,r}} \to \chi_{H_{\nu_E(x)}}$ in $L^{1}_{loc}(\R^{n})$ as $r\to 0^{+}$, where $H_{\nu_{E}(x)}$ denotes the half-space through $0$ whose exterior normal is $\nu_{E}(x)$;

\item[(iii)] for any Borel set $A$, $P(E;A) = \H^{n-1}(A\cap \de^{*}E)$, thus in particular $P(E)=\H^{n-1}(\partial^* E)$;

\item[(iv)] $\int_{E}\div g = \int_{\de^{*}E} g\cdot \nu_{E}\, d\H^{n-1}$ for any $g\in C^{1}_{c}(\R^{n};\R^{n})$.
\end{itemize}
\end{thm}

\begin{thm}[Federer's Structure Theorem]\label{thm:Fed}
Let $E$ be a set of finite perimeter. Then, $\de^* E \subset E^{(1/2)} \subset \de^e E$ and one has
\[
\H^{n-1}\left(\de^e E \setminus \de^* E\right)=0 \,.
\]
\end{thm}

In what follows, $\Omega$ will always denote a \textit{domain} of $\R^{n}$, i.e., an open connected set coinciding with its measure-theoretic interior. In other words, we assume that any point $x\in \R^{n}$, for which there exists $r>0$ with the property $|B_{r}(x)\setminus E|=0$, is necessarily contained in $\Omega$.

The next result combines \cite[Theorem 9.6.4]{Mazya2011} and \cite[Theorem 10 (a)]{AnzGia1978}.
\begin{thm}\label{thm:Mazya2011}
Let $\Omega\subset \R^{n}$ be a bounded domain with $P(\Omega) = \H^{n-1}(\de\Omega) <+\infty$. Then the following are equivalent:
\begin{itemize}
\item[(i)] there exists $k = k(\Omega)$ such that for all $E\subset \Omega$
\[
\min\{ P(E; \Omega^{\mathsf{c}}),P(\Omega \setminus E; \Omega^{\mathsf{c}}) \} \leq k P(E; \Omega);
\]

\item[(ii)] there exists a continuous trace operator from $BV(\Omega)$ to $L^{1}(\de\Omega)$ with the following property: any $\phi\in L^{1}(\de\Omega)$ is the trace of some $\Psi\in W^{1,1}(\R^{n})$ on $\de\Omega$.
\end{itemize}
\end{thm}

\begin{defin}[Cheeger constant and Cheeger set]
Let $\Omega \subset \mathbb{R}^n$ be an open, connected and bounded set. We define the \textit{Cheeger constant} of $\Omega$ as
\begin{equation}\label{eq:Cheegerconstant}
h(\Omega) := \inf \left\{P(A)/|A|\,:\ A\subset \Omega,\ |A|>0\right\}\,.
\end{equation}
Any Borel set $E\subset \Omega$ for which $P(E)/|E| = h(\Omega)$ is called a \textit{Cheeger set} of $\Omega$. 
\end{defin}

We here report some useful results on the Cheeger problem. More details are available in the survey papers \cite{Leo, Parini2011}.

\begin{prop}[Monotonicity of the Cheeger constant]\label{prop:monotonicity}
Given any two open, connected and bounded sets $\Omega_1 \subset \Omega_2$ one has $h(\Omega_1) \geq h(\Omega_2)$.
\end{prop}

\begin{thm}[Existence of Cheeger sets]\label{thm:existenceofCS}
Let $\Omega \subset \R^n$ be a bounded open set. Then the $\inf$ in \eqref{eq:Cheegerconstant} is a $\min$, therefore at least one Cheeger set $E$ for $\Omega$ exists.
\end{thm}

\begin{prop}[Properties of planar Cheeger sets]\label{prop:proprietaCS}
Let $\Omega \subset \R^2$ be an open, bounded and connected set and $E$ a Cheeger set for $\Omega$. Then the following hold
\begin{itemize}
\item[(i)] the free boundary of $E$, i.e. $\de E \cap \Omega$, is analytical and has constant curvature equal to $h(\Omega)$, hence $\de E \cap \Omega$ is a union of arcs of circle of radius $r=h^{-1}(\Omega)$;
\item[(ii)] any arc in $\de E \cap \Omega$ can not be longer than $\pi r$;
\item[(iii)] any arc in $\de E \cap \Omega$ meets tangentially $\de \Omega$ whenever they meet in a regular point of $\de \Omega$;
\item[(iv)] the volume of $E$ is bounded from below as follows
\begin{equation}\label{eq: bound on volume}
|E| \geq \pi \left (\frac{2}{h(\Omega)} \right)^2\,.
\end{equation}
\end{itemize}
\end{prop}

Notice that if $E_1,E_2$ are Cheeger sets of $\Omega$, and if $E_1\cap E_2$ is non-negligible, then one can show that $E_1\cap E_2$ is a Cheeger set (see for instance \cite[Proposition 2.5]{LeoPra}). Coupling this fact with Proposition \ref{prop:proprietaCS}(iv) one easily deduces the existence of \textit{minimal Cheeger sets} (with respect to inclusion) \textit{within any bounded open set $\Omega$}. If $\Omega$ is the unique minimizer of $h(\Omega)$ we shall say it is \emph{a minimal Cheeger set}.

\section{A minimal Cheeger set with a fat Cantor set in its boundary}\label{sec:fatCantor}

In this section we provide an example of a minimal Cheeger set, whose perimeter is strictly smaller than the $\H^{1}$-measure of its topological boundary, that is, it does not verify property \eqref{eq:PH}. We also note that, as a consequence of the construction, it is not possible to find a Lebesgue-equivalent open set for which \eqref{eq:PH} holds. Let us start noticing the following, general fact.

\begin{prop}\label{prop:selfCeD1}
If $\Omega$ is a minimal Cheeger set such that $\H^{n-1}(\Omega^{(1)}\cap \de \Omega)=0$, then $P(\Omega) = \H^{n-1}(\de \Omega)$.
\end{prop}

\begin{proof}
Being $\Omega$ a minimal Cheeger set such that $\H^{n-1}(\Omega^{(1)}\cap \de \Omega)=0$, by \cite[Theorem 3.4]{Sar17} the following relative isoperimetric inequality holds:
\[
\min\{ P(A; \Omega^{\mathsf{c}}),P(\Omega \setminus A; \Omega^{\mathsf{c}}) \} \leq k\, P(A; \Omega) \qquad \forall\, A \subset \Omega\,
\]
which in turn implies $\de \Omega \cap \Omega^{(0)} = \emptyset$, as proved in the same paper (see \cite[Lemma 3.5]{Sar17}). The thesis then follows at once by applying Theorems \ref{thm:DeGiorgi} and \ref{thm:Fed}.
\end{proof}

In virtue of Proposition \ref{prop:selfCeD1}, in order to build a minimal Cheeger set $\Omega$ that does not satisfy \eqref{eq:PH} we must ensure that the set of points of density $1$ for $\Omega$ that are also contained in $\de\Omega$ has positive $\H^{n-1}$-measure.

Consider the concentric balls $B_1,B_{\e} \subset \R^2$, where the radius $\epsilon<1$ will be fixed later on. We now define a set $F^\e \subset B_\epsilon$ whose topological boundary contains a ``fat'' Cantor set with positive $\H^{1}$-measure. Consequently, the open set $\Omega := B_1 \setminus \overline{F^\e}$ will be shown to satisfy \eqref{eq:MC}, while \eqref{eq:PH} fails.

We consider the segment $C^{\e}_{0} = [-\epsilon, \epsilon] \times \{0\}\subset \overline{B_\epsilon}$ and iteratively construct a  decreasing sequence $C^{\e}_{i}$, $i\in \N$, of compact subsets of $C^{\e}_{0}$, obtained at each step $i$ of the construction by removing $2^{i-1}$ open segments $S^{i}_{j}$, $j=1,\dots,2^{i-1}$, of length 
\[
\H^{1}(S^{i}_{j}) = 2^{1-2i}\H^{1}(C^{\e}_{i-1}),\qquad \text{for all }j,
\]
and placed in the middle of each closed segment of $C^{\e}_{i-1}$, so that the total loss of length at step $i$ equals $2^{-i}\H^{1}(C^{\e}_{i-1})$. Consequently, the set $C^{\e} = \lim_{i\to\infty}C^{\e}_{i}$ satisfies 
\[
\H^{1}(C^{\e}) = 2\e\prod_{k=1}^\infty \left(1-2^{-k} \right) >0\,.
\]
The strict positivity of the infinite product can be easily inferred by the fact that the series ${\sum_{k=1}^{\infty}\log(1 - 2^{-k})}$ is convergent. $C^{\e}$ is a so-called ``fat'' Cantor set.
  
Let now $\delta>0$ be fixed. We set
\[
f_\delta(x) = \begin{cases}
1-\sqrt{1- \left(|x|-\delta\right)^2} & \text{if }x\in (-\delta, \delta),\\
0 & \text{otherwise,}
\end{cases}
\]
and
\[
F_\delta =\{ (x,y)\in \R^2\,:\, |x|\le \delta,\ |y|\leq f_\delta(x) \}\,,
\]
which is depicted in Figure \ref{fig:fatcantor}.

\begin{figure}[t]
\centering
    \begin{minipage}{0.48\textwidth}
        \centering
       \begin{tikzpicture}[line cap=round,line join=round,>=triangle 45,x=.8cm,y=.8cm]
\clip(-4.5,-2.9) rectangle (4.5,3.1);
\draw [shift={(-4.,10.)}] plot[domain=4.71238898038469:5.123312689317197,variable=\t]({1.*10.*cos(\t r)+0.*10.*sin(\t r)},{0.*10.*cos(\t r)+1.*10.*sin(\t r)});
\draw [shift={(-4.,-10.)}] plot[domain=4.71238898038469:5.123312689317197,variable=\t]({1.*10.*cos(\t r)+0.*10.*sin(\t r)},{0.*10.*cos(\t r)+-1.*10.*sin(\t r)});
\draw [shift={(4.,10.)}] plot[domain=4.71238898038469:5.123312689317197,variable=\t]({-1.*10.*cos(\t r)+0.*10.*sin(\t r)},{0.*10.*cos(\t r)+1.*10.*sin(\t r)});
\draw [shift={(4.,-10.)}] plot[domain=4.71238898038469:5.123312689317197,variable=\t]({-1.*10.*cos(\t r)+0.*10.*sin(\t r)},{0.*10.*cos(\t r)+-1.*10.*sin(\t r)});
\draw (-0.3260526206307281,0.16660651928268677) node[anchor=north west] {$F_{\delta}$};
\begin{scriptsize}
%\draw [fill=qqqqtt] (-4.,0.) circle (0.5pt);
\draw[color=qqqqtt] (-4.125093508821784,0.25444561496340506) node {$-\delta$};
%\draw [fill=qqqqtt] (4.,0.) circle (0.5pt);
\draw[color=qqqqtt] (4.098841824285444,0.25444561496340506) node {$\delta$};
\end{scriptsize}
\end{tikzpicture}
\caption{The shape of the planar set $F_{\delta}$}
\label{fig:fatcantor}
    \end{minipage}\hfill
    \begin{minipage}{0.48\textwidth}
        \centering
        \begin{tikzpicture}[line cap=round,line join=round,>=triangle 45,x=.10cm,y=.10cm]
\clip(-13.,-30.2) rectangle (27.7,10.2);
\draw [line width=0.8pt] (-0.15350975133565697,-10.)-- (0.12808264073245562,-10.);
\draw [line width=0.8pt] (14.818808437007302,-10.)-- (14.53721604493919,-10.);
\draw [line width=0.8pt] (-3.680165577246429,-10.)-- (-3.961757969314542,-10.);
\draw [line width=0.8pt] (18.345464262918078,-10.)-- (18.627056654986188,-10.);
\draw [shift={(1.1663246714179119,13.)},line width=0.8pt]  plot[domain=4.71238898038469:4.984096843512072,variable=\t]({1.*23.*cos(\t r)+0.*23.*sin(\t r)},{0.*23.*cos(\t r)+1.*23.*sin(\t r)});
\draw [shift={(13.498974014253736,13.)},line width=0.8pt]  plot[domain=4.71238898038469:4.984096843512072,variable=\t]({-1.*23.*cos(\t r)+0.*23.*sin(\t r)},{0.*23.*cos(\t r)+1.*23.*sin(\t r)});
\draw [shift={(-2.3,-9.57984)},line width=0.8pt]  plot[domain=4.707074492493774:5.869137181301326,variable=\t]({1.*0.420165933523451*cos(\t r)+0.*0.420165933523451*sin(\t r)},{0.*0.420165933523451*cos(\t r)+1.*0.420165933523451*sin(\t r)});
\draw [shift={(-1.533675328582084,-9.57984)},line width=0.8pt]  plot[domain=4.707074492493774:5.869137181301326,variable=\t]({-1.*0.420165933523451*cos(\t r)+0.*0.420165933523451*sin(\t r)},{0.*0.420165933523451*cos(\t r)+1.*0.420165933523451*sin(\t r)});
\draw [shift={(1.1663246714179119,-33.)},line width=0.8pt]  plot[domain=4.71238898038469:4.984096843512072,variable=\t]({1.*23.*cos(\t r)+0.*23.*sin(\t r)},{0.*23.*cos(\t r)+-1.*23.*sin(\t r)});
\draw [shift={(13.498974014253736,-33.)},line width=0.8pt]  plot[domain=4.71238898038469:4.984096843512072,variable=\t]({-1.*23.*cos(\t r)+0.*23.*sin(\t r)},{0.*23.*cos(\t r)+-1.*23.*sin(\t r)});
\draw [shift={(-1.533675328582086,-10.42016)},line width=0.8pt]  plot[domain=4.707074492493774:5.869137181301326,variable=\t]({-1.*0.420165933523451*cos(\t r)+0.*0.420165933523451*sin(\t r)},{0.*0.420165933523451*cos(\t r)+-1.*0.420165933523451*sin(\t r)});
\draw [shift={(-2.3,-10.42016)},line width=0.8pt]  plot[domain=4.707074492493774:5.869137181301326,variable=\t]({1.*0.420165933523451*cos(\t r)+0.*0.420165933523451*sin(\t r)},{0.*0.420165933523451*cos(\t r)+-1.*0.420165933523451*sin(\t r)});
\draw [shift={(-2.3,-10.42016)},line width=0.8pt]  plot[domain=4.707074492493774:5.869137181301326,variable=\t]({1.*0.420165933523451*cos(\t r)+0.*0.420165933523451*sin(\t r)},{0.*0.420165933523451*cos(\t r)+-1.*0.420165933523451*sin(\t r)});
\draw [shift={(-2.3,-9.57984)},line width=0.8pt]  plot[domain=4.707074492493774:5.869137181301326,variable=\t]({1.*0.420165933523451*cos(\t r)+0.*0.420165933523451*sin(\t r)},{0.*0.420165933523451*cos(\t r)+1.*0.420165933523451*sin(\t r)});
\draw [shift={(-1.533675328582084,-10.42016)},line width=0.8pt]  plot[domain=4.707074492493774:5.869137181301326,variable=\t]({-1.*0.420165933523451*cos(\t r)+0.*0.420165933523451*sin(\t r)},{0.*0.420165933523451*cos(\t r)+-1.*0.420165933523451*sin(\t r)});
\draw [shift={(-1.533675328582084,-9.57984)},line width=0.8pt]  plot[domain=4.707074492493774:5.869137181301326,variable=\t]({-1.*0.420165933523451*cos(\t r)+0.*0.420165933523451*sin(\t r)},{0.*0.420165933523451*cos(\t r)+1.*0.420165933523451*sin(\t r)});
\draw [shift={(16.198974014253743,-9.57984)},line width=0.8pt]  plot[domain=4.707074492493774:5.869137181301326,variable=\t]({1.*0.420165933523451*cos(\t r)+0.*0.420165933523451*sin(\t r)},{0.*0.420165933523451*cos(\t r)+1.*0.420165933523451*sin(\t r)});
\draw [shift={(16.96529868567166,-9.57984)},line width=0.8pt]  plot[domain=4.707074492493774:5.869137181301326,variable=\t]({-1.*0.420165933523451*cos(\t r)+0.*0.420165933523451*sin(\t r)},{0.*0.420165933523451*cos(\t r)+1.*0.420165933523451*sin(\t r)});
\draw [shift={(16.96529868567166,-10.42016)},line width=0.8pt]  plot[domain=4.707074492493774:5.869137181301326,variable=\t]({-1.*0.420165933523451*cos(\t r)+0.*0.420165933523451*sin(\t r)},{0.*0.420165933523451*cos(\t r)+-1.*0.420165933523451*sin(\t r)});
\draw [shift={(16.198974014253743,-10.42016)},line width=0.8pt]  plot[domain=4.707074492493774:5.869137181301326,variable=\t]({1.*0.420165933523451*cos(\t r)+0.*0.420165933523451*sin(\t r)},{0.*0.420165933523451*cos(\t r)+-1.*0.420165933523451*sin(\t r)});
\draw [line width=0.8pt] (7.332649342835824,-10.) circle (2.cm);
\end{tikzpicture}
\caption{The set $\Omega_\epsilon$}
\label{fig:OmegaFatCantor}
    \end{minipage}
\end{figure}

Notice that $\de F_{\delta}$ is a union of four circular arcs of radius $1$. For $i\in \N$ we set $\delta_{i} = 2^{-2i}\H^{1}(C^{\e}_{i-1})$ and let $m^{i}_{j}$ denote the midpoint of $S^{i}_{j}$, then define  
\[
F^{\e} = \bigcup_{i\in \N}\bigcup_{j=1}^{2^{i-1}} F^{i}_{j}\,,
\]
where $F^{i}_{j} = m^{i}_{j} + F_{\delta_{i}}$. 
For $x\in [-\e,\e]$ we define
\begin{equation}\label{eq:effe}
f(x) = \sum_{i=1}^{\infty}\sum_{j=1}^{2^{i-1}} f_{\delta_{i}}(x - \mu^{i}_{j})\,,
\end{equation}
where $(\mu^{i}_{j},0) = m^{i}_{j}$. We note that $F^{\e}$ is contained in the region bounded by the graphs of $f$ and $-f$. Since $f$ is $1$-Lipschitz, $F^\e$ is necessarily contained in $\overline{B_\epsilon}$. We now define
\begin{equation}\label{es1:omega}
\Omega_\e = B_1\setminus \overline{F^{\e}}\,,
\end{equation}
whose aspect can be seen in Figure \ref{fig:OmegaFatCantor}.

\begin{prop}\label{prop:misuraperimetro}
The open set $\Omega_\e$ defined in \eqref{es1:omega} satisfies $P(\Omega_\e) < \H^1(\partial \Omega_\e)$.
\end{prop}

\begin{proof}
In general we have $P(F^{\e}) \le \H^{1}(\de F^{\e})$, therefore $P(F^{\e})$ is finite because $\H^{1}(\de F^{\e})$ is finite by construction. According to Theorem \ref{thm:DeGiorgi} we only need to show that $P(F^{\e}) = \H^{1}(\de^{*}F^{\e}) < \H^1(\partial F^{\e})$. Clearly $\partial F^{\e} = C^{\e} \cup {(F^{\e})}^{\left(1/2\right)} \cup \hat{F^{\e}}$, where $\hat F^{\e}$ is the set of corner points of $\de F^{\e}$ that do not belong to the segment $C^{\e}_{0}$. Since $\hat{F^{\e}}$ is at most countable, it has null $\H^1$-measure and therefore
\[
\H^1(\partial F^{\e}) = \H^1(C^{\e}) + \H^1\left({(F^{\e})}^{(1/2)}\right) = \H^1(C^{\e}) + \H^{1}(\de^{*}F^{\e})\,,
\]
also owing to Theorem \ref{thm:Fed}. 
The claim follows at once by recalling that $\H^{1}(C^{\e})>0$.
\end{proof}
%
%This measure is easy to compute, due to our construction. Since $C$ is what is left by each subsequent removal of intervals and since at each step we remove $1/2^{i}$ of the remaining part, one has as the remaining mass at step $i$ the quantity
%\[
%\H^1(C) =  2\epsilon \prod_{k=1}^\infty \left(1-\frac{1}{2^{k}} \right)>0\,.
%\]
%\lim_{i\to \infty} m_i$, where $m_i$ denotes the remaining mass at step $i$ and is recursively defined as
%
%which is given by\footnote{The rightmost term is known as the Pochhammer symbol in number theory.}
%\begin{equation}\label{eq:measureofC}
%\H^1 (C) = \prod_{k=1}^\infty \left(1-\frac{1}{2^k} \right) = \left( \frac{1}{2}; \frac{1}{2}\right)_\infty.
%\end{equation}
%\[
%m_i := m_{i-1} \left(1 - \frac{1}{2^i}\right) = \prod_{k=1}^i \left(1 - \frac{1}{2^k} \right), \quad \text{having set $m_0 = 2\epsilon$.}
%\]

%This product is strictly greater than zero as it can be seen by noting that its logarithm
%\[
%\log \left( \prod_{k=1}^\infty \left(1-\frac{1}{2^{k}} \right) \right) = \sum_{k=1}^\infty \log \left(1-\frac{1}{2^{k}} \right)
%\]
%converges to a finite negative number, since it is a sum of all negative terms that goes as $(-1/2^i)$. Therefore the product in ~(\ref{eq:measureofC}) must be greater than zero (namely it amounts roughly to $0,288\dots$).

Now we show that $\Omega_\e$ is a minimal Cheeger set as soon as $\e$ is small enough. The proof of this fact will be obtained through some intermediate steps. First of all, by the boundedness of  $\Omega_\e$ and by Theorem \ref{thm:existenceofCS} we know that $\Omega_\e$ admits at least a Cheeger set, from now on generically denoted as $E$. Then we have the following, intermediate result.

\begin{prop}\label{prop:proprieta}
Let $\e<1/24$ and let $\Omega_\e$ be as in \eqref{es1:omega}. Then
\begin{itemize}
\item[(i)] $h(\Omega_\e) \in \left(2, \frac{2}{1-\epsilon}\right]$;
\item[(ii)] if $E$ is a Cheeger set of $\Omega_\e$ then any connected component of $\de E \cap \Omega_\e$ is a circular arc with curvature equal to $h(\Omega_\e)$ and length less or equal than $\pi h(\Omega_\e)^{-1}$;
\item[(iii)] any Cheeger set of $\Omega_\e$ is  P-indecomposable;
\item[(iv)] the minimal Cheeger set $E_{0}$ of $\Omega_\e$ is unique, connected, and $2$-symmetric;
%\item[(v)] if $\e<1/24$ then the minimal Cheeger set $E_{0}$ has only one connected component.
\end{itemize}
\end{prop}

\begin{proof}
By the inclusions $B_1\setminus \overline{B_\epsilon} \subset \Omega_\e \subset B_1$, (i) follows from Proposition \ref{prop:monotonicity}. On the other hand (ii) follows from Proposition \ref{prop:proprietaCS} (i)-(ii). The proof of (iii) is a bit more involved. By Proposition \ref{prop:proprietaCS} (iv) we have the following lower bound for the volume of any Cheeger set $E$:
\begin{equation}\label{eq:bound on volume}
|E| \geq \pi \left( \frac{2}{h(\Omega_\e)}\right)^2 \ge \pi (1-\epsilon)^2 = \pi(1-\epsilon)^2.
\end{equation}
We now argue by contradiction supposing that $E$ is P-decomposable, so that there exist $S$ and $T$, both with positive measure, and such that $E = S\cup T$ and $P(E) = P(S) + P(T)$. Then $S$ and $T$ are both Cheeger sets of $\Omega_\e$ (see for instance \cite{Parini2011}), hence they must satisfy \eqref{eq:bound on volume}. Since $\epsilon<1/4$ we obtain $|E| = |S|+|T| > 18\pi/16 > \pi = |B_1|$, which is clearly not possible. 
In order to prove (iv) we notice that, thanks to the symmetry of $\Omega_\e$, the reflection $\widetilde E_{0}$ of $E_{0}$ with respect to one of the two coordinate axes is a Cheeger set of $\Omega_\e$, too. By the lower bound on the volume one has $|E_{0}\cap \widetilde E_{0}|>0$, then by well-known properties of Cheeger sets, such intersection is also a Cheeger set of $\Omega_\e$. Therefore by minimality of $E_{0}$ we infer $E_{0} = E_{0}\cap \widetilde E_{0} = \widetilde E_{0}$, which shows the claimed symmetry of $E_{0}$. Notice moreover that, by the same argument, $E_{0}$ is unique.
In order to show the topological connectedness of $E_0$, we can suppose by contradiction, and without loss of generality, that there are just two connected components $E_1, E_2$ of $E_{0}$, and that $E_2$ is obtained by reflecting $E_1$ with respect to one of the axes of symmetry of $\Omega_\e$. By (iii) we must have $P(E_{0}) < P(E_1) +P(E_2) = 2P(E_1)$. Moreover the strict inequality implies that $\H^{1}(\de^{*} E_{1}\cap C^{\e}_{0})>0$, so that we obtain
\begin{equation}\label{eq:splitper}
2P(E_{1}) \le P(E_{0}) + 2\H^{1}(C^{\e}_{0}) = P(E_{0}) + 4\e\,.
\end{equation}
Hence by \eqref{eq:splitper} and the isoperimetric inequality we infer
\begin{align}
\frac{4}{1-\epsilon}|E_1| &\geq 2h(\Omega_\e)|E_1| = h(\Omega_\e)|E_{0}| = P(E_{0})\nonumber \\
&\geq 2P(E_1) -4\epsilon \geq 4\sqrt{\pi}|E_1|^{\frac{1}{2}} - 4\epsilon = 4\sqrt{\frac{\pi}{2}} |E_{0}|^{1/2} - 4\epsilon \nonumber \\
&\geq \frac{4\pi}{\sqrt{2}}(1-\epsilon) - 4\epsilon. \nonumber
\end{align}
Then if $\epsilon<1/24$ we find
\[
|E_{0}| = 2|E_1| \geq \sqrt{2}\pi(1-\epsilon)^2 - 2\epsilon(1-\epsilon) > \pi,
\]
that is, a contradiction.
\end{proof}

\begin{thm}\label{thm:cantor}
Let $\e<1/24$. Then, $\Omega_\e$ defined in \eqref{es1:omega} is a minimal Cheeger set.
\end{thm}

\begin{proof}
Let $E_{0}$ be a minimal Cheeger set of $\Omega_\e$. By Proposition \ref{prop:proprieta} (iv) we know that $E_{0}$ is $2$-symmetric and unique. Assume now by contradiction that $E_{0}$ does not coincide with $\Omega_\e$. This implies that $\de E_{0}\cap \Omega_\e \neq \emptyset$, thus there exists at least one connected component of $\de E_{0}\cap \Omega_\e$ consisting of a circular arc $\alpha$ of radius $r=h(\Omega_\e)^{-1}$, whose endpoints $p,q$ necessarily belong to $\de \Omega_\e$. We now rule out all possibilities depending on where the endpoints $p$ and $q$ are located. This will be accomplished by the discussion of the following four cases (hereafter we adopt the same notation introduced in the proof of Proposition \ref{prop:misuraperimetro}, i.e., we denote by $\hat F^{\e}$ the set of corner points of $\de F^{\e}$ that do not belong to $C^{\e}_{0}$).

\textit{Case 1: one of the endpoints of $\alpha$ belongs to $\de B_{1}$.} Let us assume without loss of generality that $p\in \de B_{1}$. In this case we have to distinguish two subcases. First, if $q\in \de B_{1}$ then $\alpha$ must touch $\de B_{1}$ in a tangential way at both $p$ and $q$, however the radius $r$ is smaller than $1/2$, so that necessarily $p=q$, that is, $\alpha$ is a full circle, which is in contrast with Proposition \ref{prop:proprieta} (ii). Second, if $q\in \de F^{\e}$, the arc $\alpha$ can be symmetric neither with respect to the $x$-axis nor with respect to the $y$-axis. Therefore, by symmetry, $\de E_{0}\cap \Omega_\e$ has at least three more other connected components. These components cannot touch, but in the endpoints. Then, there exist at least two connected components of $E_{0}$, which yields a contradiction with Proposition \ref{prop:proprieta} (iv).

\textit{Case 2: one of the endpoints of $\alpha$ belongs to $\de^{*} F^{\e}$.} We can assume that $p\in \de^{*} F^{\e}$ and $q\in \de F^{\e}$. In this case the arc $\alpha$ is contained in the closure of the ball of radius $1$ that is tangent to $\de^{*}F^{\e}$ at $p$ and does not intersect $F^{\e}$ (by construction of $F^{\e}$ there is exactly one such ball for any $p\in \de^{*}F^{\e}$). Consequently the only possibility is that $p=q$, which is not possible as discussed in Case 1.

\textit{Case 3: $p$ and $q$ belong to the fat Cantor set $C^{\e}$.} By the assumption on $\e$ coupled with Proposition \ref{prop:proprieta} (i) we infer that $r=h(\Omega_\e)^{-1} > 2\e$. Then we observe that $\alpha$ is the smaller arc cut by the chord $\overline{pq}$ on one of the two possible circles of radius $r$ passing through both $p$ and $q$. We finally have that $\alpha \subset B_{\e}$ and thus $E_{0}$ has a connected component $E_{0}'$ contained in $B_\e$, but this is not possible as by  \eqref{eq:bound on volume} and the choice of $\e$ we have 
\[
\pi\e^{2} \ge |E_{0}'| \ge \pi\left(\frac{2}{h(\Omega_\e)}\right)^{2} \ge \pi(1-\e)^{2}\,.
\]

\textit{Case 4: one endpoint belongs to $\hat F^{\e}$, the other to $\hat F^{\e}\cup C^{\e}$.} As before we can assume without loss of generality that $p$ is a corner point on the graph of $f$, where $f$ is defined in \eqref{eq:effe}, and that $q\in \hat F^{\e}\cup C^{\e}$. Notice that $q$ must belong to the upper half-plane, otherwise $\alpha$ would cross the segment $C^{\e}_{0}$. This means that $q$ belongs to the graph of $f$ over $[-\e,\e]$. Moreover, the curvature vector associated with $\alpha$ at $p$ must have a positive component with respect to the $y$-axis, otherwise we would fall into the same situation of Case 3 (i.e., the presence of a too small connected component of $E_{0}$). Consequently, by comparing the graph of $f$ (whose generalized curvature is bounded from above by $1$) with the arc $\alpha$ (whose curvature is $h(\Omega_\e) \ge 2$) we deduce by the maximum principle that their intersection can only contain $p$, which contradicts the fact that $q$ belongs to that intersection. This concludes the discussion of Case 4, and thus the proof of the theorem.
\end{proof}

It is natural to ask whether solutions $u_\e^\pm$ of \eqref{eq:PMC} with $\Omega = \Omega_\e$ and $H(x) = H_\e^\pm$ exist, for the two prescribed mean curvatures defined as
\[
H_\e^- = P(\Omega_\e) / |\Omega_\e|\qquad\text{and}\qquad 
H_\e^+ = \H^1(\de \Omega_\e) / |\Omega_\e| = \big(P(\Omega_\e) + \H^1(C^\e)\big) / |\Omega_\e|\,.
\]
One can thus consider two approximating sequences of sets, $\{\Omega_{\e,j}^-\}_j$ and $\{\Omega_{\e,j}^+\}_j$, defined in the following way. The first sequence, $\{\Omega_{\e,j}^-\}_j$, is monotone decreasing towards $\Omega_\e$ and is obtained by subsequently removing each rescaled and translated copy of $F_\delta$ from the ball $B_1$. The second sequence, $\{\Omega_{\e,j}^+\}_j$, is monotone increasing and constructed by removing smaller and smaller tubular neighborhoods of $\overline{F^\e}$ from $B_1$. Clearly, both sequences converge to $\Omega_\e$ in the $L^1$ sense, however only the first one converges also in the perimeter sense, as $j\to\infty$. It can be shown that $\Omega_{\e,j}^\pm$ is a minimal Cheeger set, for all $j$ large enough. Now, the idea is to define 
\[
H^\pm_{\e,j} = P(\Omega_{\e,j}^\pm) / |\Omega_{\e,j}^\pm| 
\]
and to solve \eqref{eq:PMC} on $\Omega_{\e,j}^\pm$ with $H = H^\pm_{\e,j}$, thus obtaining two sequences of solutions $u_{\e,j}^\pm$ that, up to suitable vertical translations, and relying on the theory of generalized solutions as described in \cite{Miranda1977} (see also \cite{Giusti1978}), will converge to some limit functions $u_\e^\pm$. Then, $u_\e^\pm$ will be solutions of \eqref{eq:PMC} on $\Omega_\e$ for $H = H_\e^\pm$, respectively. Notice that both $u_\e^-$ and $u_\e^+$ become vertical at the reduced boundary of $\Omega_\e$. This shows that $\Omega_\e$ provides a counterexample to the possibility of extending the characterization of existence and uniqueness up to vertical translations, that has been proved in \cite[Theorem 4.1]{LS16a} under the assumption of weak regularity of the domain.

\section{A minimal Cheeger set with fast-decaying porosity near its boundary}\label{sec:emmenthal}

In this section we provide an example of set $\Omega_{\bold 0}\subset \R^{2}$ that is a minimal Cheeger set, i.e. it satisfies \eqref{eq:MC}, and whose perimeter $P(\Omega_{\bold 0})$ equals $\H^{1}(\de\Omega_{\bold 0})$. Its peculiarity is that there is a subset $A$ of the reduced boundary $\de^* \Omega_{\bold 0}$ with $\H^1(A) >0$, such that $\de \Omega_{\bold 0}$ is not locally a graph at any point $x\in A$.

We define the set $J$ of pairs $\mathbf{j} = (j_{1},j_{2})$ such that $j_{1},j_{2}\in \N$ and $j_{2}\le j_{1}$, then for any $\bj\in J$ we set
\begin{equation*}
\mathbf{j}+1=
\begin{cases}
(j_{1}+1,1) \qquad &\text{if $j_{2}=j_{1}$,}\\
(j_{1},j_{2}+1) &\text{if $j_{2}<j_{1}$.}
\end{cases}
\end{equation*}
We fix two sequences $(\e_{\mathbf{j}})_{\mathbf{j}\in J}$ and $(r_{\mathbf{j}})_{\mathbf{j}\in J}$ of positive real numbers between $0$ and $\frac 12$, that will be specified later, and define
\begin{alignat*}{5}
&\rho_{\bj} && = 1 - \epsilon_{\bj}, \hspace{5cm} &&\theta_{\bj} &&= j_{2}\cdot \frac{\pi}{2(j_{1}+1)},\\ 
&x_{\bj} &&= \rho_{\bj} \left(\cos( \theta_{\bj}), \sin( \theta_{\bj})\right), &&B_{\bj} &&= B_{r_{\bj}}(x_{\bj})\,,
\end{alignat*}
so that in particular $x_{\bj}$ is a point of $B_1 = B_{1}(0)$ contained in the first quadrant, for all $\bj \in J$. We write $\bj \preceq \mathbf{j'}$ (or equivalently $\mathbf{j'}\succeq \bj$) if $\bj$ precedes or is equal to $\mathbf{j'}$ with respect to the standard lexicographic order on $J$. The notion of ``limit as $\bj\to \infty$'' is the obvious one associated with this order relation. We require the following properties on the sequences introduced above:
\begin{itemize}
\item[(i)] $\sum_\bj r_\bj \le 1/{(2^8 +1)}$;
\item[(ii)] $\e_{\mathbf{1}} < 1/4$;
\item[(iii)] $\e_{\bj+1}  \le \frac{3}{10}\e_\bj$;
\item[(iv)] $r_\bj \le 2^{-18}\e_\bj^{3}$.
\end{itemize}
Notice that (iii) and (iv) imply that $\e_\bj - 2\e_{\bj +1} \ge r_{\bj} +2r_{\bj+1}$. This in turn implies that the closures of the balls $B_{r_{\bj}}(x_{\bj})$ are pairwise disjoint.
We then set 
\begin{equation}\label{defOmegaPoroso}
\Omega_{\bold 0} := B_1 \setminus  \bigcup \limits_{\bj\succeq \bold 0} \overline{B_{\bj}},
\end{equation}
which is an open set since the only accumulation points of the sequence of ``holes'' $B_{\bj}$ are contained in $\de B_1$. Sequential zoom-ups of how this set is, can be seen in Figure \ref{fig:poroso}.
\begin{figure}[t]
\centering
\includegraphics[trim={1.5cm 5cm 1.5cm 4.5cm},clip, width=\textwidth]{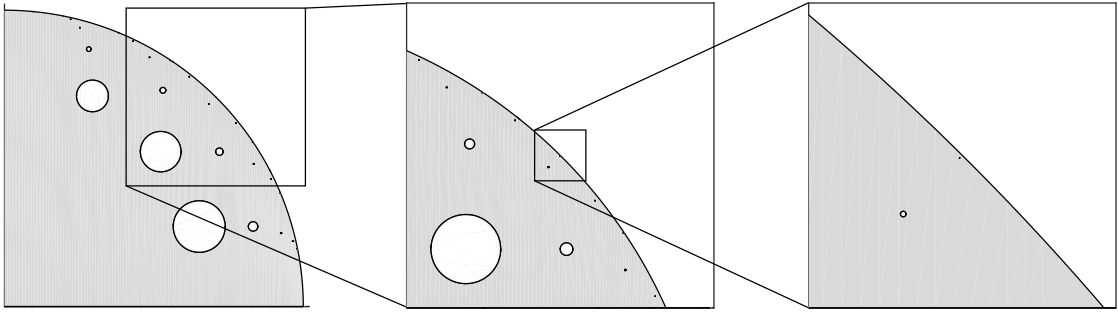} 
\caption{Close-ups of the set $\Omega_{\bold 0}$ of Section \ref{sec:emmenthal}.
}
\label{fig:poroso}
\end{figure}
Once proved that this set is a minimal Cheeger set, it is quite easy to build from it a sequence of minimal Cheeger sets converging to the unitary ball both in volume and in perimeter by ``filling'' the holes one at a time. Let indeed $\Omega_{\bold k} $ be the set defined by
\[
\Omega_{\bold k} := B_1 \setminus \bigcup \limits_{\bj \succeq \bold k} \overline{B_{\bj}}\,.
\]
Clearly $\Omega_{\bold k} \subset \Omega_{\bold h}$ whenever $\bold k \preceq \bold h$, and as $\bold k \to \infty$ the sequence $\Omega_{\bold k} \to B_1$ both in perimeter and area. It is clear that their Cheeger constants converge to that of the unit disk $B_1$. Therefore, one can apply the stability result for solutions of the prescribed mean curvature equation proved in \cite[Proposition 4.4]{LS16a} to this sequence of domains.

Before dealing with the minimality of $\Omega_{\bold 0}$, we show that the topological boundary $\de\Omega_{\bold 0}$ coincides with the reduced boundary $\de^{*}\Omega_{\bold 0}$.

\begin{prop}\label{porosoSR}
Under the above assumptions (i)-(iv) one has $\de \Omega_{\bold 0} = \de^{*}\Omega_{\bold 0}$.
\end{prop}

\begin{proof}
Of course $\de^{*}\Omega_{\bold 0}\subseteq \de\Omega_{\bold 0}$. In order to prove the opposite inclusion we fix $y\in \de\Omega_{\bold 0}$ and argue as follows. If $y \in \de B_{\bj}$ for some $\bj\in J$, or $y\in \de B_1 \setminus \{ z=(z_{1},z_{2})\in \R^2\, :\, z_{1}\ge 0,z_{2}\ge 0\}$, then there exists a neighborhood $U_{y}$ of $y$ such that $\de\Omega_{\bold 0} \cap U_{y}$ is an arc of $\de B_1$ or $\de B_{\bj}$, hence trivially $y\in \de^{*}\Omega_{\bold 0}$. Assume now that $y\in \de B_1$ with non-negative coordinates $y_{1},y_{2}$. It is standard to check that, in this case, $y\in \de^{*}\Omega_{\bold 0}$ if and only if 
\begin{equation}\label{eq:stimaridotta}
P(\Omega_{\bold 0}; B_{s}(y)) \le 2s + o(s),\quad s\to 0\,. 
\end{equation}
In order to show \eqref{eq:stimaridotta} we first set 
\[
J_{2}(j_{1},s) = \Big\{j_{2}\in \{1,\dots,j_{1}\}:\ |x_{\bj}-y|<s+r_{\bj}<2s\Big\}\,.
\]
Then there exists a least index $j_{1}(s)\in \N$ such that $J_{2}(j_{1},s)$ is empty whenever $j_{1}<j_{1}(s)$, while in general we obtain 
\begin{equation}\label{eq:cardJ2}
\# J_{2}(j_{1},s) \le 1+\frac{32(j_{1}+1)s}{\pi}\quad \text{when $j_{1}\ge j_{1}(s)$.} 
\end{equation}
To prove this estimate on the cardinality of $J_{2}(j_{1},s)$ we observe that for $\bj=(j_{1},j_{2})$ and $\mathbf{j'} = (j_{1},j_{2}')$ belonging to $J_{2}(j_{1},s)$ we have
\begin{equation}\label{eq:chain01}
\frac{1}{2} \Big|(\cos \theta_{\bj}-\cos \theta_{\mathbf{j'}}, \sin\theta_{\bj} - \sin\theta_{\mathbf{j'}})\Big| \le |x_{\bj}-x_{\mathbf{j'}}| \le |x_{\bj}-y| + |x_{\mathbf{j'}}-y| < 4s\,,
\end{equation}
where for the first inequality we have also used the fact that $|x_{\bj}|> \frac 12$ for all $\bj$. Then, setting 
\[
h = |\theta_{\bj}-\theta_{\mathbf{j'}}| = \frac{|j_{2}' - j_{2}|\pi}{2(j_{1}+1)}
\]
one easily obtains from \eqref{eq:chain01} that
\[
\sin h \le \left|(\cos \theta_{\bj}-\cos \theta_{\mathbf{j'}}, \sin\theta_{\bj} - \sin\theta_{\mathbf{j'}})\right| <8s\,,
\]
whence assuming $s< \frac 1{16}$ one deduces 
\[
h \le 16s\,,
\]
which implies $|j_{2}-j_{2}'| \le 32(j_{1}+1)s/\pi$. Then \eqref{eq:cardJ2} follows at once. In conclusion we find 
\begin{align*}
P(\Omega_{\bold 0}; B_{s}(y)) &= 2s+o(s) + P\left(\bigcup_{\bj\in J}B_{\bj}; B_{s}(y)\right) \le 2s+o(s) + \sum_{j_{1}=1}^{\infty} \sum_{j_{2}\in J_{2}(j_{1},s)} 2\pi r_{\bj}\\
&\le 2s+o(s) + s\sum_{j_{1}=j_{1}(s)}^{\infty} [2\pi + 64(j_{1}+1)] r_{(j_{1},1)} = 2s+o(s)
\end{align*}
where the last equality relies on the fact that
\[
kr_{(k,1)} \le k \e_{(k,1)}^{3} \le k\e_{\mathbf 1}^{3} \left(\frac{3}{10}\right)^{3(k^{2}-k)/2}
\]
which follows by (ii), (iii) and (iv). This latter says that the sum converges.
\end{proof}

By Theorem \ref{thm:existenceofCS}, $\Omega_{\bold 0}$ admits at least one Cheeger set. We will denote by $E$ a Cheeger set of $\Omega_{\bold 0}$. The main goal now is to show that, necessarily, $E = \Omega_{\bold 0}$.

\begin{thm}\label{selfCheegernessOmega}
Let $\e_{\bj}$ and $r_{\bj}$ be such that (i)-(iv) hold. Then, $\Omega_{\bold 0}$ is a minimal Cheeger set.
\end{thm}
The proof of Theorem \ref{selfCheegernessOmega} will require some preliminary results. We start by defining the following quantity
\[
\delta = \frac{1+\sum_{\bj}r_{\bj}}{1-\sum_{\bj}r_{\bj}^{2}} - 1\,,
\]
which will be used later on.

\begin{prop}\label{prop:defdelta}
Let $\Omega_{\bold 0}$ be defined as in \eqref{defOmegaPoroso} and let $E$ be a Cheeger set of $\Omega_{\bold 0}$. Assume that (i)-(iv) hold. Then, 
\begin{align}\label{stimePerEmm-1}
2\le h(\Omega_{\bold 0}) \le 2(1+\delta),\\\label{stimePerEmm-2}
|E| \geq \frac{\pi}{(1+\delta)^{2}}\,.
\end{align}
\end{prop}

\begin{proof}
The first inequality in \eqref{stimePerEmm-1} follows directly from the inclusion $\Omega_{\bold 0}\subset B$ and from Proposition \ref{prop:proprietaCS}, while the second is a consequence of $h(\Omega_{\bold 0}) \le \frac{P(\Omega_{\bold 0})}{|\Omega_{\bold 0}|}$. Then \eqref{stimePerEmm-2} follows from \eqref{eq: bound on volume} at once.
\end{proof}
%%%%%%%%%

Notice that (i) implies $\delta < 1/2^7$. Indeed let $\eta = \sum_\bj r_\bj$. Then, since $\eta > \sum_\bj r_\bj^2$ one has
\begin{equation}\label{eq:bounddelta}
\delta = \frac{1+\sum_{\bj}r_{\bj}}{1-\sum_{\bj}r_{\bj}^{2}} - 1  \leq \frac{1+\eta}{1-\eta} - 1 \le \frac{1}{2^7}\,.
\end{equation}
Thus, by Proposition \ref{prop:defdelta} we have
\begin{equation}\label{stimasuhperdeltascelto}
2\le h(\Omega_{\bold 0}) \le 2(1+\delta) < 3\,.
\end{equation}

\begin{lem}\label{lem:2.1}
Let $\Gamma$ be an arc swept by a disk of radius $r< 1/2$ contained in an annulus of inner and outer radii equal to, respectively, $1/2$ and $1$. Denote by $o$ the center of the annulus and by $a,b$ the endpoints of $\Gamma$. If the region $R$ enclosed by (the vectors) $a,b$ and $\Gamma$ is convex then 
\[
|p| \geq \min\{ |a|, |b|\}\qquad \forall\, p \in \Gamma\,.
\]
\end{lem}

\begin{proof}
The configuration described in the statement is depicted in Figure \ref{fig:lemma21}. To prove the lemma we argue by contradiction and suppose that there exists $p_0 \in \Gamma \setminus \{a, b\}$ such that
\[
|p_{0}| = \min_{p\in \Gamma} |p|< \min\{|a|, |b|\}\,.
\]
If we denote by $c$ the center of the disk sweeping the arc $\Gamma$, by minimality of $p_0$ we have that $p_0, c, o$ lie on the same line. Moreover, being the region $R$ convex by our assumption, we infer that $c$ and $o$ lie on the same half-plane cut by the tangent in $p_0$ to $\Gamma$. We now claim that $c$ lies in between $o$ and $p_{0}$. If this were not the case one would have $|p_{0}-c| > |p_0|$ which in turn implies $r>1/2$ against our hypotheses. Therefore we have $|p_0-c| +|c| = |p_0|$ and by the triangular inequality
\[
|a| \leq |c|+|c-a| = |c|+|p_0-c| = |p_0|,
\]
against our initial assumption.
\end{proof}

\begin{figure}[t]
\centering
\begin{tikzpicture}[line cap=round,line join=round,>=triangle 45,x=.5cm,y=.5cm]
\clip(-6.5,-6.5) rectangle (6.5,6.5);
\fill[color=qqzzqq,fill=qqzzqq,fill opacity=0.1] (0.33046314381140673,-1.5994797005893902) -- (2.980463143811406,-1.0594797005893901) -- (0.,0.) -- (-0.04953685618859338,-4.259479700589383) -- cycle;
\draw(0.,0.) circle (3.cm);
\draw(0.,0.) circle (1.5cm);
\draw [shift={(0.33046314381140673,-1.5994797005893902)}] plot[domain=-1.7126933813990615:0.2010213597695852,variable=\t]({1.*2.6870057685088735*cos(\t r)+0.*2.6870057685088735*sin(\t r)},{0.*2.6870057685088735*cos(\t r)+1.*2.6870057685088735*sin(\t r)});
\draw [dash pattern=on 3pt off 3pt] (0.,0.)-- (0.8741340098383927,-4.230909348030764);
\draw [dotted] (-7.372690047545156,-5.934758045490492)-- (9.005395739465408,-2.550936596056631);
\draw (0.,0.)-- (-0.04953685618859338,-4.259479700589383);
\draw (0.,0.)-- (2.980463143811406,-1.0594797005893901);
\draw [shift={(0.33046314381140673,-1.5994797005893902)},color=qqzzqq,fill=qqzzqq,fill opacity=0.1]  (0,0) --  plot[domain=-1.7126933813990615:0.2010213597695852,variable=\t]({1.*2.6870057685088735*cos(\t r)+0.*2.6870057685088735*sin(\t r)},{0.*2.6870057685088735*cos(\t r)+1.*2.6870057685088735*sin(\t r)}) -- cycle ;
\begin{scriptsize}
\draw (-3.8,6) node[anchor=north west] {$B_1$};
\draw (-3,3.4) node[anchor=north west] {$B_{\frac12}$};
\draw (2.8,-2.5) node[anchor=north west] {$\Gamma$};
\draw (0.5,-0.8) node[anchor=north west] {$c$};
\draw[color=black] (-0.3323399724962303,0.14393750253352733) node {$o$};
\draw[color=black] (-0.4,-4.2) node {$a$};
\draw[color=black] (3.3,-1.) node {$b$};
\draw[color=uuuuuu] (1.1,-4.8) node {$p_0$};
\end{scriptsize}
\end{tikzpicture}
\caption{The configuration of Lemma \ref{lem:2.1}.}
\label{fig:lemma21}
\end{figure}
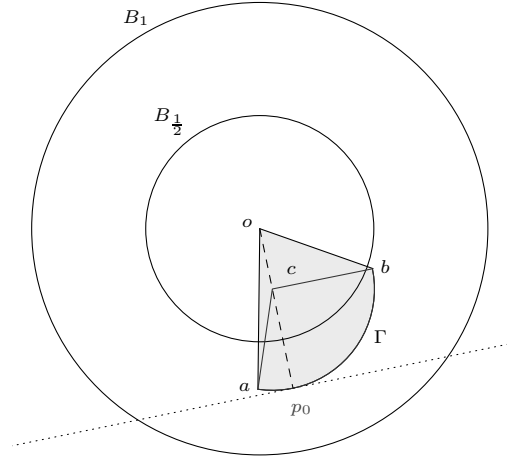

\begin{lem}[Density estimate]\label{lem:upperdensity}
Let $E$ be a Cheeger set of $A\subset \R^{2}$. Fix $z\in A$ and $r>0$ such that $B_{r}(z) \subset A$. Then 
\begin{equation}\label{stimadensita}
|B_{r}(z)\setminus E| \le \pi r^{2}/36 \quad\Rightarrow\quad B_{2r/3}(z) \subset E\,.
\end{equation}
\end{lem}
\begin{proof}
Let us set $m(r) = |B_{r}(z)\setminus E|$ and define $F = E \cup B_{r}(z)$ as a competitor. The minimality of $E$ implies that 
\begin{align*}
\frac{P(E)}{|E|} &\le \frac{P(F)}{|F|} = \frac{P(E, \R^{2}\setminus \overline{B_{r}(z)}) + m'(r)}{|E|+m(r)}\\ 
&= \frac{P(E) - P(B_{r}(z)\setminus E) + 2m'(r)}{|E|+m(r)}
\end{align*}
for almost all $r>0$, hence
\begin{align*}
\frac{P(E)}{|E|}m(r) + P(B_{r}(z)\setminus E) &\le 2m'(r)\,.
\end{align*}
In particular we find that $P(B_{r}(z)\setminus E) \le 2m'(r)$, therefore by the isoperimetric inequality in $\R^{2}$ we obtain
\begin{equation}\label{diffineq1}
m'(r) \ge \sqrt{\pi} m(r)^{\frac 12}\,.
\end{equation}
If we now assume by contradiction that $m(2r/3)>0$ then we can integrate the differential inequality
\[
\frac{m'(t)}{m(t)^{\frac 12}} \ge \sqrt{\pi}
\]
between $2r/3$ and $r$, thus obtaining
\[
0<m(2r/3)^{\frac 12} \le m(r)^{\frac 12} - \frac{\sqrt{\pi r^{2}}}{6} \le 0\,,
\]
that is a contradiction.
\end{proof}

\begin{lem}\label{lem:pallaunmezzo}
Let $\Omega_{\bold 0}$ be constructed as before. If (i)-(iv) hold, then the disk $B_{1/2}$ is contained in any Cheeger set $E$ of $\Omega_{\bold 0}$.
\end{lem}
\begin{proof}
By \eqref{stimePerEmm-2} and \eqref{eq:bounddelta} we have that 
\[
|B_{3/4}\setminus E| \le |B_{1}| - |E| \le \pi - \frac{\pi}{(1+\delta)^{2}} = 	\frac{2+\delta}{(1+\delta)^{2}}\pi\delta \le \frac{2+\delta}{1+\delta}\pi\delta \le 2\pi\delta \le \frac{\pi (3/4)^{2}}{36}\,,
\]
hence we can apply Lemma \ref{lem:upperdensity} and obtain that $B_{1/2} = B_{\frac{2}{3}\cdot\frac{3}{4}}\subset \Omega_{\bold 0}$ is also contained in $E$.
\end{proof}

Let us fix a Cheeger set $E$ of $\Omega_{\bold 0}$ and assume that $\partial E\cap \Omega_{\bold 0} \neq \emptyset$. Then we consider the (at most countable) collection $\{\Gamma_{k}\}_{k\in \N}$ of the closures of the connected components of $\partial E\cap \Omega_{\bold 0}$. Notice that $\Gamma_{k}$ is a closed circular arc of radius $r = h(\Omega_{\bold 0})^{-1}$.

We observe that $\cup_{k} \Gamma_{k}$ is locally compact in $B_{1}$, as only a finite number of arcs can have a nonempty intersection with $B_{t}$, for all $0<t<1$. Then, we have the following result.

\begin{lem}\label{lem:pzeroend}
Assume (i)-(iv) and that $\de E\cap \Omega_{\bold 0} \neq\emptyset$. Denote by $p_0$ a point of $\cup_{k} \Gamma_{k}$ minimizing the distance from the origin. Then there exists $k_0$ such that $p_0$ is one of the endpoints of $\Gamma_{k_0}$.
\end{lem}
\begin{proof}
Since $\cup_{k} \Gamma_{k}\cap B_{1}$ is nonempty and locally compact in $B_{1}$, there exists $k_{0}\in \N$ such that $p_{0}\in \Gamma_{k_{0}}$. Assume now by contradiction that $p_{0}$ is not one of the endpoints $a_{0},b_{0}$ of $\Gamma_{k_{0}}$, then owing to Lemma \ref{lem:pallaunmezzo}, $B_{1/2}\subset E$. Thus by Lemma \ref{lem:2.1}, the region enclosed by $\Gamma_{k_{0}}$ and the segments connecting $a_{0}$ and $b_{0}$ to the origin cannot be convex. Therefore, since $B_{1/2} \subset E$, the segment $\sigma_{0}$ connecting $p_{0}$ to the origin must intersect the boundary of $E$ at some first point $q_{0}$ strictly closer than $p_{0}$ to the origin. Indeed, the Cheeger set locally lies on the convex side of $\Gamma_{k_{0}}$ near $p_{0}$. To conclude we need to exclude the possibility that $q_{0}\in \partial \Omega_{\bold 0}\setminus \partial B_{1}$, which means that $q_{0}\in \partial B_\bj$ for some $\bj$. Let now consider the shortest of the two closed arcs of $\partial B_\bj$ cut by $\sigma_{0}$ (note that the arc could degenerate to a single point), and call it $\gamma$. Notice that all the points of $\gamma$ have a distance from the origin which is strictly less than $|p_{0}|$. Then $\gamma$ must contain at least an endpoint of some $\Gamma_{k}$, otherwise there would exist an open neighbourhood $U$ of $\gamma$ such that $U\cap \partial E\cap \Omega_{\bold 0} = \emptyset$, but this cannot hold as $U$ must contain points of $E$ (this comes from the fact that $q_{0}\in \gamma$) as well as points of $\Omega_{\bold 0}\setminus E$ (this is a consequence of the fact that the connected component of $\sigma_{0}\cap \Omega_{\bold 0}$ having an endpoint on $\partial B_\bj$, and being the closest to $p_{0}$, is made of points of $\Omega_{\bold 0}\setminus E$). Therefore $q_{0}\in \Omega_{\bold 0}$, hence $q_{0}\in \Gamma_{k}$ for some $k$, which contradicts the minimality of $p_{0}$. This concludes the proof. 
\end{proof}

\begin{lem}\label{lem:lemma25}
Assume (i)-(iv) and let $p_{0}$ be as in Lemma \ref{lem:pzeroend}. Then letting $\alpha$ be the angle spanned by the half-tangent to $\Gamma_{k_0}$ in $p_0$ and the segment connecting $p_0$ to the origin, one has
\begin{equation}\label{alfaottuso}
\alpha > \frac{\pi}{2} + \frac{d_{0}}{2}\,,
\end{equation}
where $d_{0} = \dist(p_{0},\partial B_{1})$.
\end{lem}
\begin{proof}
Let $B_\bj$ be the ball whose boundary contains $p_{0}$. Let $p_{1}$ be the second endpoint of $\Gamma_{k_{0}}$ and denote by $p_{*}$ the point of $\Gamma_{k_{0}}$ minimizing the distance from $\partial B_{1}$. Since $p_{1}\in \partial \Omega_{\bold 0}$, by construction of $\Omega_{\bold 0}$ we infer that either $p_{1}\in \partial B_{1}$, or $p_{1}\in \partial B_{\mathbf{j'}}$ with $\bj \prec \mathbf{j'}$, therefore the distance $d_{*} = \dist(p_{*},\partial B_{1})$ must satisfy $d_{*}< d_{0}/2$. 
%%%\footredGP{in generale, ordine $\e_{\bj} - \e_{\bj +1}$. Ad esempio se $\e_{\bj} = \frac{1}{|\bj| + 10}$ si ha ordine $1/|j|^{2} \sim \e_{\bj}^{2}$ ovvero $d^{*}\le d_{0}-c\e_{\bj}^{2}$}
Indeed this holds true if $\e_{\bj} - 2\e_{\bj+1} \ge r_{\bj}+2r_{\bj+1}$, which follows from conditions (vii) and (viii).
Let $c$ be the center of the arc $\Gamma_{k_{0}}$ and consider the triangle $T$ with vertices $p_{0},c$ and the origin. Notice that $|p_{0}-c| = r < 1/2$ and $|p_{0}| = 1-d_{0}$ while by the triangular inequality applied to the triangle $T^*$ of vertices $p^*, c, o$ we have
\[
|c| \ge |p^{*}| - r = 1 - r - d^{*} \ge 1-r-d_{0}/2\,.
\]
Moreover if we assume that $\alpha < \pi$ (otherwise the estimate would be trivial) then the internal angles of $T$ at $p_{0}$ and at the origin (respectively, $\gamma$ and $\beta$) are smaller than $\pi/2$. Indeed for $\alpha < \pi$ we find that 
\[
\langle p_{0}, \nu_{\bj}(p_{0})\rangle <0\,,
\] 
where $\nu_{\bj}(p_{0})$ denotes the outer normal to $\partial B_{\bj}$ at $p_{0}$, thus $\alpha> \pi/2$.  Then, $\gamma = \alpha - \pi/2 \in [0, \pi/2)$. Finally, $|p_{0}|>r$, whence $\beta < \pi/2$ as claimed. Consequently the orthogonal projection $z$ of $c$ onto the line through the opposite side of $T$ must lie between the origin and $p_{0}$, that is, $|p_{0}| = |z|+|p_{0}-z|$. Then we have
\begin{align*}
|c|^{2} - |z|^{2} &= r^{2} - |p_{0}-z|^{2}\,,
\end{align*}
whence by rearranging terms 
\begin{align*}
|c|^{2} -  r^{2}&= |z|^{2} - |p_{0}-z|^{2}\\ 
&= |p_{0}| \cdot (|z|-|p_{0}-z|)\\ 
&= |p_{0}| \cdot (|p_{0}|-2|p_{0}-z|)\\ 
&= (1-d_{0})(1-d_{0}-2|p_{0}-z|)\,.
\end{align*}
On the other hand
\[
|c|^{2}-r^{2} \ge (1-r-d_{0}/2)^{2} -r^{2} = 1+d_{0}^{2}/4 -2r -d_{0} +d_{0}r\,,
\]
thus we find
\[
2|p_{0}-z| \le 1-d_{0} - \frac{1+d_{0}^{2}/4 -(2-d_{0})r -d_{0}}{1-d_{0}}\,.
\]
Consequently we have
%%%\footredGP{In generale si ottiene $\cos \gamma \le 1-C\e_{\bj}^{2}$ con $C$ costante stimabile uniformemente.}
\begin{align*}
\cos \gamma = \frac{|p_{0}-z|}{r} &\le \frac{2r(1-d_{0}) +rd_{0}-d_{0}+3d_{0}^{2}/4}{2r(1-d_{0})}\\ 
&= 1 - \frac{d_{0}(1-r) - 3d_{0}^{2}/4}{2r(1-d_{0})} < 1 - d_{0}/4\,,
\end{align*}
where the last inequality follows as soon as $d_{0}<1/3$. Being $d_0 \le \e_{\mathbf{1}} + r_{\mathbf{1}}$, this condition is met thanks to (ii) and (iii). Then, we have
\[
\sin^{2}\gamma = 1 - \cos^{2}\gamma > 1 - (1-d_{0}/4)^{2} = d_{0}/2 - d_{0}^{2}/4 >  d^2_{0}/4
\]
and thus we conclude that
\[
\gamma > \sin \gamma >  {d_{0}}/2\,.
\]
Since $\alpha = \pi/2 + \gamma$, we get \eqref{alfaottuso}.
\end{proof}

\begin{lem}\label{lem:2.6}
Assume (i)-(iv) and let $p_{0}$, $\Gamma_{k_0}$, $d_{0}$ and $\alpha$ be as in Lemma \ref{lem:lemma25}. Let $p \in \Gamma_{k_0}$ be a point such that $0<|p_0-p| < d_0/12$. Then, denoting by $\eta$ the angle in $p_0$ spanned by the half-tangent to $\Gamma_{k_0}$ at $p_0$ and the segment from $p_0$ to $p$, one has
\begin{equation}\label{boundoneta}
\xi := \alpha - \eta > \frac{\pi}{2} + \frac{d_{0}}{4}\,.
\end{equation}
\end{lem}

\begin{proof}
Let $c$ be the center of the disk sweeping $\Gamma_{k_0}$ and let $h$ be the projection of $p$ onto the half-tangent to $\Gamma_{k_0}$ at $p_0$. Since $\xi = \alpha -\eta$, by Lemma \ref{lem:lemma25}, it is enough to provide an upper bound for $\eta$.

To this aim we consider the triangles $T$ of vertices $p_0, p_h$ and $h$ and $S$ of vertices $p_0, c$ and $m$, where $m$ is the midpoint of the segment $p-p_0$, as in Figure \ref{fig:lemma26}. It is easy to see they are similar with angles $\pi/2, \eta,$ and $\pi/2 - \eta$. Therefore we have the proportionality relation
\begin{equation}
\frac{|p-h|}{|p-p_0|}= \frac{|p-p_{0}|}{2r}\,, \nonumber
\end{equation}
whence by recalling that $0<\eta<\pi/2$ and that $r> 1/3$ by \eqref{stimePerEmm-1} and the condition on $\delta$ one obtains
\begin{equation}\label{stimePera2}
\frac{\eta}2 \le \sin(\eta) = \frac{|p-h|}{|p-p_0|} = \frac{|p-p_0|}{2r} < \frac{d_0}{24r} < \frac{d_0}{8}\,.
\end{equation}
This upper bound on $\eta$ combined with \eqref{alfaottuso} yields the claim.
\end{proof}

\begin{figure}[t]
\centering
\begin{tikzpicture}[line cap=round,line join=round,>=triangle 45,x=1.0cm,y=1.0cm]
\clip(-0.5,-3.2) rectangle (5.5,6.);
%\fill[color=uuuuuu,fill=uuuuuu,fill opacity=0.1] (0.,0.) -- (1.3972901743470358,-1.147504908503453) -- (4.964846377279386,3.1966257281919384) -- cycle;
%\fill[color=uuuuuu,fill=uuuuuu,fill opacity=0.1] (0.,0.) -- (1.8587452345795503,-2.886914305544091) -- (2.7873215898217665,-2.2890486633478204) -- cycle;
%\draw [shift={(0.,0.)},color=qqwuqq,fill=qqwuqq,fill opacity=0.1] (0,0) -- (-57.224463476866475:0.5050250027931988) arc (-57.224463476866475:-39.394069771231734:0.5050250027931988) -- cycle;
\draw [shift={(0.,0.)},color=wwqqzz,fill=qqzzqq,fill opacity=0.1] (0,0) -- (-39.394069771231734:0.5050250027931988) arc (-39.394069771231734:90.:0.5050250027931988) -- cycle;
\draw (0.,5.623009713745612)-- (0.,0.);
\draw [shift={(4.964846377279386,3.1966257281919384)}] plot[domain=3.7136336800261276:4.334516827282664,variable=\t]({1.*5.904922996629413*cos(\t r)+0.*5.904922996629413*sin(\t r)},{0.*5.904922996629413*cos(\t r)+1.*5.904922996629413*sin(\t r)});
\draw [dash pattern=on 2pt off 2pt] (4.964846377279386,3.1966257281919384)-- (0.,0.);
\draw [dash pattern=on 2pt off 2pt] (2.7873215898217665,-2.2890486633478204)-- (4.964846377279386,3.1966257281919384);
\draw (0.,0.)-- (2.7873215898217665,-2.2890486633478204);
\draw (0.,0.)-- (1.8587452345795503,-2.886914305544091);
\draw (2.7873215898217665,-2.2890486633478204)-- (1.8587452345795503,-2.886914305544091);
\draw [dash pattern=on 2pt off 2pt] (4.964846377279386,3.1966257281919384)-- (1.3972901743470358,-1.147504908503453);
\begin{scriptsize}
\draw (1.8,-2.1) node {$\Gamma$};
%\draw (1.3564174152410364,-1.5820136594373433) node[anchor=north west] {$T$};
%\draw (1.440588249039903,0.6400963528527268) node[anchor=north west] {$S$};
\draw[color=ttqqqq] (0.11068907501781239,5.799768464723231) node {$o$};
\draw[color=uuuuuu] (-0.18390884327822032,-0.08377281781752333) node {$p_0$};
\draw[color=uuuuuu] (5.076768269150935,3.3756484513158815) node {$c$};
\draw[color=uuuuuu] (3,-2.3) node {$p$};
\draw[color=uuuuuu] (1.4,-2.8) node {$h$};
\draw[color=uuuuuu] (1.7,-1.2) node {$m$};
\draw[color=qqwuqq] (4.2,2.5) node {$\eta$};
\draw[color=wwqqzz] (0.2,0.1) node {$\xi$};
\end{scriptsize}
\end{tikzpicture}
\caption{The configuration of Lemma \ref{lem:2.6}.}
\label{fig:lemma26}
\end{figure}
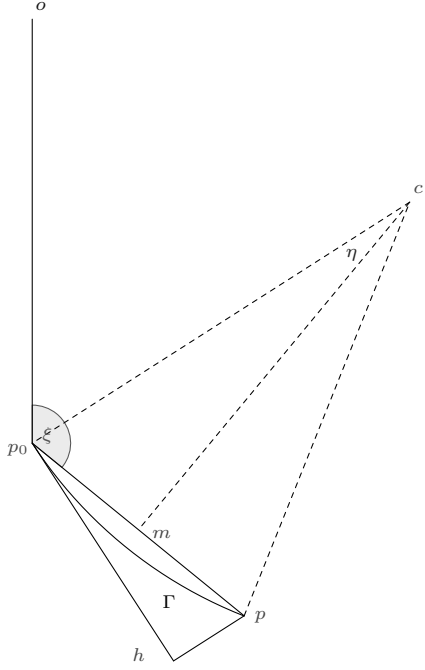

\begin{rem}
Note that Lemmas \ref{lem:lemma25} and \ref{lem:2.6} hold whenever $p_0$ is the endpoint of an arc $\Gamma$ such that $p_0$ minimizes $|p|$ among $p\in \Gamma$.
\end{rem}

\subsection{Proof of Theorem \ref{selfCheegernessOmega}}

We remark that it would not be too difficult to apply a compactness argument and show that, for a suitable choice of parameters, the set $\Omega^{\bj}$ defined as
\[
\Omega^{\bj} := B_1 \setminus  \bigcup \limits_{\bi \preceq \bj} \overline{B_\bi}\,,
\]
is a minimal Cheeger set for all $\bj$. Then, by passing to the limit as $\bj \to \infty$ and by exploiting Theorem 2.7 of \cite{LeoPra}, we would infer that $\Omega_{\bold 0}$ is a Cheeger set as well. However, this simple argument tells us nothing about the uniqueness of the Cheeger set of $\Omega_{\bold 0}$. In other words, there seems to be no way of deducing that $\Omega_{\bold 0} = \lim_{\bj}\Omega^{\bj}$ is a minimal Cheeger set from the minimality of $\Omega^{\bj}$. This is due to the lack of uniform a-priori estimates in the spirit of the quantitative isoperimetric inequality (see in particular \cite{CicLeo2012,CicLeo2013}). In this specific case, the existence of a modulus of continuity $\phi$ independent of $\bj$, such that 
\[
P(E)/|E| - h(\Omega^{\bj}) \ge \phi(|\Omega^{\bj}\setminus E|)
\]
for all $\bj$ and all measurable $E\subset \Omega^{\bj}$, would be needed. By an application of the selection principle introduced in \cite{CicLeo2012} we could obtain $\phi = \phi_{\bj}$, however it is not clear how to exclude a possible degeneracy of the sequence $\{\phi_{\bj}\}_{\bj}$, as $\bj\to\infty$. 
Therefore we choose to follow what reveals to be a much more involved and technically complex path leading to a direct proof of uniqueness. Indeed, by combining the various, intermediate lemmas proved before we ultimately show that any Cheeger set $E$ of $\Omega_{\bold 0}$ must necessarily satisfy $\de E \cap \Omega_{\bold 0}=\emptyset$. Owing to the connectedness of $\Omega_{\bold 0}$ and the fact that $B_{1/2}\subset E$, this is sufficient to conclude that $E = \Omega_{\bold 0}$. Before delving into the proof, we remark that there are four different kinds of arcs inside $\de E \cap \Omega_{\bold 0}$, depending on where their endpoints lie:
\begin{itemize}
\item[(a)] arcs $\Gamma$ with both endpoints on $\partial B_1$;
\item[(b)] arcs $\Gamma$ with both endpoints on $\partial B_{\bj}$ for some $\bj$;
\item[(c)] arcs $\Gamma$ with an endpoint of $\partial B_1$ and one of $\partial B_{\bj}$ for some $\bj$;
\item[(d)] arcs $\Gamma$ with an endpoint on $\partial B_{\bj}$ and one on $\partial B_{\bi}$ with $\bj \neq \bi$.
\end{itemize}
While cases (a) and (b) can be easily excluded by property (ii) of Proposition \ref{prop:proprietaCS}, cases (c) and (d) are much trickier. For these latter two cases the argument is actually the same: we will build a competitor that has a smaller Cheeger ratio, thus contradicting the minimality of $E$. In order to do so, we will also employ Lemma \ref{lem:2.6}.

\begin{proof}[Proof of Theorem \ref{selfCheegernessOmega}]
Argue by contradiction and suppose $\de E \cap \Omega_{\bold 0} \neq \emptyset$.

\emph{Step 1.} We start by showing that cases (a) and (b) cannot happen. Let $\Gamma$ be the arc with endpoints $p,q \in \de B_{\bj}$. Being these points regular, by Proposition \ref{prop:proprietaCS} (iii) the arc $\Gamma$ must be tangent to $B_{\bj}$ in both points. By Proposition \ref{prop:defdelta} and the choice of $r_\bj$  the curvature of $B_{\bj}$ is strictly greater than the curvature of $\Gamma$. Therefore one necessarily has that points $p$ and $q$ coincide which implies that $\Gamma$ is a full circle which contradicts property (ii) of Proposition \ref{prop:proprietaCS}. An analogue reasoning holds for an arc $\Gamma$ with endpoints $p,q \in \de B_1$.

\emph{Step 2.} We now show that cases (c) and (d) cannot happen. We will exhibit a competitor to $E$ that has a better Cheeger ratio against the minimality of $E$. Pick the point $p_0$ provided by Lemma \ref{lem:pzeroend} and consider the arc $\Gamma_{p_0}$ with endpoint $p_0$. There exists a pair $\bj$ such that $p_0 \in \de B_{\bj}$. Trivially there exists at least another point $q_0$ on the boundary of $B_{\bj}$ from which another arc of $\de E\cap \Omega_{\bold 0}$ departs. Let $z \in \de B_\bj$ be the ``north pole'', i.e. the closest point to the origin. Note that there is only a finite number of arcs of $\de E\cap \Omega_{\bold 0}$ touching $\de B_{\bj}$. Moreover, since $|p_{0}|>r$ we find that $|p_{0}| >|z|$ (otherwise we would have $p_{0}=z$ and this would contradict the fact that $p_{0}$ minimizes the distance of points of $\Gamma_{p_{0}}$ from the origin). This shows that $z$ is contained in a connected component $\psi$ of $\de B_{\bj} \setminus \mathcal E_{\bj}$, where $\mathcal E_{\bj}$ denotes the (finite) set of endpoints of arcs of $\de E\cap \Omega_{\bold 0}$ that lie on $\de B_{\bj}$. One of the endpoints of $\psi$ is, of course, $p_{0}$. Let $q_{0}$ denote the other endpoint belonging to the arc $\Gamma_{q_{0}}$.

From now on we shall assume that $\psi$ is smaller than a half-circle, otherwise the construction of the competitor would be even easier. 

Since $p_0$ minimizes the distance of $\de E \cap \Omega_{\bold 0}$ from the origin we have that
\[
d_{q_0} := \dist (q_0, \de B_1) \leq \dist (p_0, \de B_1) =: d_{p_0}.
\]
We now fix two points $q \in \Gamma_{q_0}$ and $p\in \Gamma_{p_0}$ such that
\begin{equation}\label{sceltapeq}
|p-p_0| = |q-q_0| = \frac{d_{q_0}}{16}.
% > \frac{3}{2}\pi r_{\bj}
\end{equation}
%where the last inequality holds if $\e_\bj > (24\pi +1)r_\bj$ which is met due to condition (viii). 
We can apply Lemma \ref{lem:2.6} to the couples of points $p,p_0$ and $q, q_0$ obtaining the estimate from below of the angles $\xi_q$ and $\xi_p$ (that correspond to $\xi$ in Lemma \ref{lem:2.6}):
\begin{equation}\label{angolixi}
\xi_{q}, \xi_{p} > \frac{\pi}{2} +  \frac{d_{q_0}}{4}.
\end{equation}
We now modify the Cheeger set $E$ into $\widetilde E$ by adding the region delimited by $\de B_{\bj}, \Gamma_{q_0}, \Gamma_{p_0}$ and the segment $p-q$. To contradict the minimality of $E$ it is enough to show that $\delta P = P(\widetilde E) - P(E) < 0$ for $\epsilon$ small enough. It is straightforward that
\begin{equation}\label{stimadP}
\delta P \leq 2\pi r_{\bj} - |p-p_0|-|q-q_0| + |p-q| = 2\pi r_{\bj} - 2|p-p_0| + |p-q|.
\end{equation}

\begin{figure}[t]
\centering
\begin{tikzpicture}[line cap=round,line join=round,>=triangle 45,x=1.0cm,y=1.0cm]
\clip(-5.,-4.5) rectangle (0.5,0.5);
\fill[color=wwccff,fill=qqzzqq,fill opacity=0.1] (-3.1683974519907685,-1.5478749262626132) -- (-4.570643537792309,-3.712250913654666) -- (-2.9866367014187247,-3.9505156690964034) -- (-3.0971296935521333,-1.5780502980808755) -- cycle;
\draw(-3.1391518266208713,-1.5780502980808755) circle (0.04202213306873811cm);
\draw (0.,0.)-- (-3.0971296935521333,-1.5780502980808755);
\draw (0.,0.)-- (-3.1683974519907685,-1.5478749262626132);
\draw (-3.0971296935521333,-1.5780502980808755)-- (-3.1683974519907685,-1.5478749262626132);
\draw [shift={(-9.455587640563849,0.9890128948132362)}] plot[domain=4.875754052335625:5.899664434885997,variable=\t]({1.*6.779716829190773*cos(\t r)+0.*6.779716829190773*sin(\t r)},{0.*6.779716829190773*cos(\t r)+1.*6.779716829190773*sin(\t r)});
\draw [shift={(1.7632060515890602,-2.5404952330550867)}] plot[domain=2.9461014151705305:3.7476766116814364,variable=\t]({1.*4.954711273964883*cos(\t r)+0.*4.954711273964883*sin(\t r)},{0.*4.954711273964883*cos(\t r)+1.*4.954711273964883*sin(\t r)});
\draw (-3.1683974519907685,-1.5478749262626132)-- (-4.570643537792309,-3.712250913654666);
\draw (-3.0971296935521333,-1.5780502980808755)-- (-2.9866367014187247,-3.9505156690964034);
\draw (-2.9866367014187247,-3.9505156690964034)-- (-4.570643537792309,-3.712250913654666);
\begin{scriptsize}
\draw[color=black] (0.1,0.1) node {$o$};
\draw[color=ttttqq] (-3.5,-1.8) node {$q_0$};
\draw[color=black] (-2.6,-2) node {$p_0$};
\draw[color=black] (-4.9,-3.7) node {$q$};
\draw[color=black] (-2.6,-3.9) node {$p$};
\end{scriptsize}
\end{tikzpicture}
\caption{The way the competitor is build.}
\label{fig:competitor}
\end{figure}
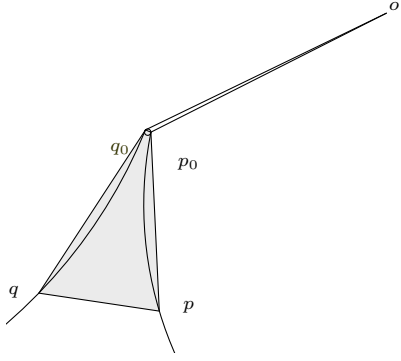

Therefore we need to estimate $|p-q|$ from above. In order to do so, we will employ the angles of the isosceles trapezoid with vertices $p_0, q_0, q, p$ (and, respectively, angles $\gamma_0$ and $\gamma$) and the triangle $T$ of vertices $o, p_0, q_0$ (and, respectively, angles $\sigma, \alpha, \beta$), denoted as in Figure \ref{fig:competitor}. We then have 
 \begin{subequations}
\begin{numcases}{\mbox{}}
  \gamma_0 +\gamma = \pi \label{angoli1}\\
  \alpha +\beta + \xi_{q} +\xi_{p} +2\gamma_0 = 4\pi \label{angoli2}\\
  \alpha+\beta +\sigma = \pi  \label{angoli3}
\end{numcases}
  \end{subequations}
where \eqref{angoli1} denotes the (half of the) sum of interior angles of the trapezoid, \eqref{angoli2} the sum of the angles in $p_0$ and in $q_0$, and \eqref{angoli3} the sum of the interior angles of the triangle $T$.

Subtracting \eqref{angoli3} to \eqref{angoli2}, and combining the resulting equality with \eqref{angolixi}  we find
\[
2\gamma_0 < 2\pi +\sigma - \frac{d_{q_0}}{2}
\]
which coupled with \eqref{angoli1} gives
\[
\gamma > \frac{d_{q_0}}{4} - \frac{\sigma}{2}\,.
\]
We now estimate $\sigma$ from above as follows. First notice that its sine is small
\[
\sin(\sigma) = \frac{|p_0-q_0|}{1-d_{q_0}}\sin(\alpha) \leq 4r_\bj \leq 2^{-4}\e_\bj\,,
\]
where the last inequality is guaranteed by (viii). Thus $\sigma$ itself is small, i.e.
\[
\frac \sigma 2 \le \sigma - \frac{\sigma^3}{6} \le \sin \sigma \le 2^{-4}\e_\bj \le 2^{-3} d_{q_0}\,,
\]
eventually getting the lower bound
\[
\gamma > \frac{d_{q_0}}{8}\,.
\]
Since $|p-q|> |p_0 -q_0|$, the angle $\gamma$ is smaller than $\pi/2$, thus
\begin{equation}\label{eq:thelatter1}
0\le \cos \gamma \le \cos \left( \frac{d_{q_0}}{8} \right) \le 1- \frac{d_{q_0}^2}{2^7} + \frac{d_{q_0}^4}{3 \cdot 2^{15}} \le 1 - \frac{d_{q_0}^2}{2^8}\,.
\end{equation}
From \eqref{sceltapeq}, \eqref{stimadP} and \eqref{eq:thelatter1} it follows that
\begin{align*}
\delta P &\le 2\pi r_{\bj} -2|p-p_{0}| + |p-q| \\
&\le 2\pi r_{\bj} -2|p-p_{0}| + 2|p-p_{0}|\cos \gamma + 2 r_{\bj}\\
&\le 2r_{\bj}(\pi +1) +\frac{d_{q_0}}{2^3} ( \cos (\gamma) - 1) \le 2r_{\bj}(\pi +1) - \frac{d_{q_0}^3}{2^{11}}
\end{align*}
Since by (viii) we have $r_\bj \le 2^{-18} \e_\bj^{3}$ and $d_{q_0} \ge  \e_\bj/2$, we obtain
\[
\frac{d^3_{q_0}}{2^{11}} \ge \frac{\e_\bj^3}{2^{14}} \ge 16r_\bj > 2r_\bj(\pi +1)\,,
\]
thus $\delta P<0$, a contradiction. This concludes the proof of the theorem.
\end{proof}

%%%%%%%%%Come si fa a far comparire il DOI per un articolo al momento solo online?

\bibliographystyle{plain}

\bibliography{LeoSar_ex}

\end{document}